\documentclass[a4paper,11pt]{amsart}

\usepackage{epsfig}
\usepackage{amsthm,amsfonts}
\usepackage{amssymb,graphicx,color}
\usepackage{float}
\graphicspath{ {Images/} }
\usepackage[labelfont={bf},textfont={it}]{caption}
\usepackage[labelfont={bf, it}]{subcaption}
\usepackage[all]{xy}
\usepackage{verbatim}
\usepackage{hyperref}
\usepackage{enumitem}
\usepackage{tikz-cd}

\newtheorem{theorem}{Theorem}[section]
\newtheorem*{theorem*}{Theorem}

\newtheorem{corollary}[theorem]{Corollary}
\newtheorem{proposition}[theorem]{Proposition}
\newtheorem{definition}[theorem]{Definition}

\newtheorem{question}{Question}
\newtheorem{claim}{Claim}[theorem]

\newtheorem{innercustomthm}{Theorem}
\newenvironment{customthm}[1]
  {\renewcommand\theinnercustomthm{#1}\innercustomthm}
  {\endinnercustomthm}

\newcommand{\R}{\mathbb{R}}

\newcommand{\N}{\mathbb{N}}

\makeatletter
\@namedef{subjclassname@2020}{%
  \textup{2020} Mathematics Subject Classification}
\makeatother

\begin{document}

\title[Local vs. global Lipschitz geometry]
{Local vs. global Lipschitz geometry}

\author[J. E. Sampaio]{Jos\'e Edson Sampaio}
\address{Jos\'e Edson Sampaio: Departamento de Matem\'atica, Universidade Federal do Cear\'a, Av. Humberto Monte, s/n Campus do Pici - Bloco 914, 60455-760, Fortaleza-CE, Brazil. E-mail: {\tt edsonsampaio@mat.ufc.br}}

\thanks{The author was partially supported by CNPq-Brazil grant 310438/2021-7. This work was supported by the Serrapilheira Institute (grant number Serra -- R-2110-39576).}


\begin{abstract}
In this article, we prove that for a definable set in an o-minimal structure with connected link (at 0 or infinity), the inner distance of the link is equivalent to the inner distance of the set restricted to the link. With this result, we obtain several consequences. We present also several relations between the local and the global Lipschitz geometry of singularities. For instance, we prove that two sets in Euclidean spaces, not necessarily definable in an o-minimal structure, are outer lipeomorphic if and only if their stereographic modifications are outer lipeomorphic if and only if their inversions are outer lipeomorphic.
\end{abstract}

\maketitle
\tableofcontents
\section{Introduction}
The local study of Lipschitz geometry of singularities is a well-established field of study that has been very active in the last 25 years.
The interest in the Lipschitz geometry of singularities at infinity is more recent, but it has been gaining a lot of attention in the last 5 years, for example, we can cite \cite{BirbrairFSV:2020,BobadillaFS:2018,CostaGM:2023,DiasR:2022,FernandesJ:2023,FernandesJS:2022,FernandesJS:2023,FernandesS:2020,FernandesS:2022,FernandesS:2022b,FernandesS:2023,Jelonek:2021,KernerPR:2018,Nhan:2023b,Sampaio:2019,Sampaio:2020,Sampaio:2023,Sampaio:2023b,SampaioS:2022,SampaioS:2023,SampaioS:2023b,Targino:2020}.

The goal of this paper is to present some relations between the local and the global Lipschitz geometry of singularities.

An example of such a relation was presented by Fernandes and the author of this article in \cite{FernandesS:2022}. By using the ideas of \cite{Birbrair}, the authors of \cite{FernandesS:2022} presented a global classification of semi-algebraic surfaces with isolated singularities under bi-Lipschitz homeomorphisms, concerning its inner distance (so-called inner lipeomorphims). As a consequence, they obtained the following result (see definitions of stereographic modification and stereographic compactification of a set in Section \ref{sec:outer_geo}):
\begin{theorem}[Corollary 5.7 in \cite{FernandesS:2022}]
	\label{cor:compactification_eq}
Let $X\subset\R^n$ and $Y\subset\R^m$ be closed semi-algebraic surfaces with isolated inner Lipschitz singularities. Then, $X$ and $Y$ are inner lipeomorphic if and only if the pointed stereographic compactifications $(\widehat{X},e_{n+1})$ and $(\widehat{Y},e_{m+1})$ are inner lipeomorphic.
\end{theorem}

Several of the relations between the local and the global Lipschitz geometry of singularities, which are presented here, are direct consequences of the following result, which is the main result of this article:

\begin{customthm}{\ref*{main_theorem}}
Let $A\subset \R^n$ be a definable set in an o-minimal structure $\mathcal{S}$.
\begin{itemize} 
 \item [(a)] If the link of $A$ at infinity is connected, then there are constants $K,r\geq 1$ such that for each $t\in (r,+\infty)$, we have
$$
d_{A,inn}(x,y)\leq d_{A_{\varphi,t},inn}(x,y)\leq K d_{A,inn} (x,y),
$$
for all $x,y\in A_{\varphi,t}$. 
\item [(b)] If the link of $A$ at $0$ is connected, then there are constants $K,r\geq 1$ such that for each $t\in (0,\frac{1}{r})$, we have
$$
d_{A,inn}(x,y)\leq d_{A_{\varphi,t},inn}(x,y)\leq K d_{A,inn} (x,y),
$$
for all $x,y\in A_{\varphi,t}$. 
\end{itemize}
\end{customthm}
The set $A_{\varphi,t}$ above is the set $ A_{\varphi,t}=\{x\in A; \varphi(x)=t\}$ and $\varphi\colon A\to \R$ is a Lipschitz and definable function in $\mathcal{S}$ such that there is a constant $C\geq 1$ satisfying $\frac{1}{C}\|x\|\leq \varphi(x)\leq C\|x\|$ for all $x\in A$. Such a function is called a {\bf radius function for $A$}.

A set $X\subset \R^n$ is link Lipschitz normally embedded (LLNE) at infinity, if there is a constant $C\geq 1$ such that $d_{X_t, inn}\leq C\|\cdot \|$, for all large enough $t>0$, where $X_t:=X\cap \mathbb{S}_t^{n-1}(0)$. This notion was introduced in \cite{FernandesS:2022b} and a priori it depended on the centre of the ball. 
So, we had the following question: 
\begin{question}\label{question:dependency_centre}
If $X\subset \R^n$ is LLNE at infinity and $p\in \R^n\setminus\{0\}$, is there a constant $C\geq 1$ such that $X\cap \mathbb{S}_t^{n-1}(p)$ is $C$-LNE, for all large enough $t>0$?
\end{question}

In Section  \ref{sec:char_lne}, consequences of Theorem \ref{main_theorem}. In particular, we obtain a positive answer to the above question (see Corollary \ref{cor:lne_llne_equiv}). 
Another consequence is the following characterization of Lipschitz normally embeddedness at infinity (see definitions of LNE sets in Subsection \ref{subsec:lne_thm} and Section \ref{sec:char_lne}):
A definable set $A\subset \R^n$ with connected link at infinity is LNE at infinity if and only if $A$ is LLNE at infinity (see Corollary \ref{cor:lne_llne_equiv}).
In particular, we recover the main results proved in \cite{mendessampaio} and \cite{Nhan:2023}. 
We recover also the results proved in \cite{CostaGM:2023}, which are essentially included in Corollaries \ref{cor:lne_llne_equiv} and \ref{cor:charac_LNE_infinity} in the particular case of closed subsets of $\R^n$.

In Section \ref{sec:key_results}, we prove two useful results. For instance, in Subsection \ref{subsec:lne_thm}, we prove the Lipschitz normal embedding theorem, which states that any connected definable set is definably inner lipeomorphic to an LNE set (see definition of lipeomorphism in Subsection \ref{subsec:lne_thm}). This result was proved in \cite{BirbrairM:2000} in the case that the set was compact and semi-algebraic. We show here that the same proof of  \cite{BirbrairM:2000} works, with small changes, in our context. And in Subsection \ref{subsec:arc-criterion_Lip}, for definable sets $X\subset\R^n$ and $Y\subset\R^m$ and $\sigma,\tilde\sigma\in \{inn,out\}$, we present a criterion for a definable mapping $\varphi\colon (X,d_{X,\sigma})\to (Y,d_{X,\tilde\sigma})$ to be Lipschitz. 

In Section \ref{sec:outer_geo}, we prove that two sets in Euclidean spaces, not necessarily definable in an o-minimal structure, are outer lipeomorphic if and only if their stereographic modifications are outer lipeomorphic if and only if their inversions are outer lipeomorphic. 
(see Theorem \ref{thm:charac_outer_geo_lip}).

In Section \ref{sec:out_inn_geo}, we obtain also a generalization of  Theorem \ref{cor:compactification_eq}. For instance, for connected definable sets $X\subset\R^n$ and $Y\subset\R^m$ and $\sigma,\tilde\sigma\in \{inn,out\}$, we prove that $(X,d_{X,\sigma})$ and $(Y,d_{X,\tilde\sigma})$ are lipeomorphic if and only if the pointed stereographic modifications $(\widehat{X},d_{\widehat{X},\sigma},\infty)$ and $(\widehat{Y},d_{\widehat{Y},\tilde\sigma},\infty)$ are lipeomorphic (see Corollary \ref{cor:global_compactification_eq_two}). In particular, we obtain that the definable (inner or outer) Lipschitz geometry of definable sets is the same in the local and in the global cases.

\bigskip

\noindent {\bf Notation:} 
\begin{itemize}
 \item $\|(x_1,...,x_n)\|=(x_1^2+...+x_n^2)^{\frac{1}{2}}$;
 \item $\mathbb{S}_r^{n-1}(p)=\{x\in \R^n; \|x-p\|=r\}$, $ \mathbb{S}_r^{n-1}=\mathbb{S}_r^{n-1}(0)$ and $ \mathbb{S}^{n-1}=\mathbb{S}_1^{n-1}(0)$;
 \item $B_r^{n}(p)=\{x\in \R^n; \|x-p\|<r\}$;
 \item Let $f,g\colon (a,+\infty)\to (0,+\infty)$ be functions. We write $f\lesssim g$ if there is a constant $f(t)\leq C g(t)$ for all $t\in (a,+\infty)$. We write $f\approx g$ if $f\lesssim g$ and $g\lesssim f$. We write $f\ll g$ if $\lim\limits_{t\to +\infty} \frac{f(t)}{g(t)}=0$;
\item We fix an o-minimal structure $\mathcal{S}$ on $\R$ (see more about o-minimal structures in \cite{Coste:1999} and \cite{Dries:1998}). So, a definable set or a definable function means definable in $\mathcal{S}$.  
  
\end{itemize}

\section{Main result}

Given a path connected subset $X\subset\R^n$, the
{\bf inner distance} on $X$ is defined as follows: given two points $x_1,x_2\in X$, $d_{X,inn}(x_1,x_2)$  is the infimum of the lengths of paths on $X$ connecting $x_1$ to $x_2$. We denote by $d_{X,out}$ the Euclidean distance of $\R^n$ restricted to $X$.

Let $A\subset \mathbb{R}^n$ be a connected definable set in $\mathcal{S}$. By the Pancake Decomposition Theorem \cite[Theorem 1.3]{KurdykaP:2006}, there are $M\geq 1$ and a partition into a definable and finite union $A=\bigcup_{i\in I}{\mathcal{B}_i}$ such that every subset $\mathcal{B}_{i}$ is $M$-LNE. For each $i$, let $X_i$ be the closure of $\mathcal{B}_i$ in $A$. We consider $x,y\in A$ and for each $r$ and $k$ we define  
$$\widetilde{\Delta}_r(x,y):=\inf\left\{\sum_{i=0}^{r-1}{\|x_{i+1}-x_i\|};x_0=x, x_r=y, x_i,x_{i+1}\in X_{\nu_i}, 0\leq i\leq r-1\right\}.$$

$$\Delta_k(x,y):=\inf\left\{\widetilde{\Delta}_r(x,y);r=1,\cdots,k\right\}.$$
We also define $\inf \emptyset=+\infty$. It is clear that every $\Delta_k$ is definable in $\mathcal{S}$. Finally, we define 
$$d_{A,P}(x,y):=\inf\left\{\Delta_k(x,y);k\in \mathbb{N} \right\}.$$ 

The proof of the next result is an adaptation of Lemma 4 and Theorem 1 in \cite{KurdykaO:1997}.

\begin{proposition}\label{prop:pancake_distance}
	The function $d_{A,P}\colon A\times A\rightarrow \mathbb{R}$ is definable in $\mathcal{S}$, defines a distance in $A$ and 
	$$d_{A,P}(x,y)\leq d_{A,in}(x,y)\leq M d_{A,P}(x,y),$$
for all $x,y\in A$.
\end{proposition}
The distance $d_{A,P}$ is called a {\bf pancake distance of $A$}.

Now, we are ready to state and prove our main result.

\begin{theorem}\label{main_theorem}
Let $A\subset \R^n$ be a definable set in an o-minimal structure $\mathcal{S}$. Let $\varphi\colon A\to \R$ be a radius function for $A$ and $A_{\varphi,t}=\{x\in A;\varphi(x)=t\}$ for each $t>0$.
\begin{itemize} 
 \item [(a)] If the link of $A$ at infinity is connected, then there are constants $K,r\geq 1$ such that for each $t\in (r,+\infty)$, we have
$$
d_{A,inn}(x,y)\leq d_{A_{\varphi,t},inn}(x,y)\leq K d_{A,inn} (x,y),
$$
for all $x,y\in A_{\varphi,t}$. 
\item [(b)] If the link of $A$ at $0$ is connected, then there are constants $K,r\geq 1$ such that for each $t\in (0,\frac{1}{r})$, we have
$$
d_{A,inn}(x,y)\leq d_{A_{\varphi,t},inn}(x,y)\leq K d_{A,inn} (x,y),
$$
for all $x,y\in A_{\varphi,t}$. 
\end{itemize}
\end{theorem}
\begin{proof}
Let us prove Item (a).

%

Let $\mathcal{A}=Graph(\varphi)=\{(x,t)\in \R^n\times \R; x\in A $ and $\varphi(x)=t\}$. We have that $\mathcal{A}$ is definable in $\mathcal{S}$ and by the Pancake Decomposition Theorem \cite[Theorem 1.3]{KurdykaP:2006}, there are a constant $C\geq 1$ and a finite definable partition $\mathcal{A}=\bigcup\limits_{i=1}^k \mathcal{B}_i$ such that for each $i\in \{1,...,k\}$, $\mathcal{B}_i\cap ( \R^n\times \{t\})$ are $C$-LNE for all $t\in \R$. 

We have that the projection $\pi\colon \R^n\times \R\to \R^n$ satisfies the following: $\pi|_{\mathcal{A}}\colon \mathcal{A}\to A$ is an outer lipeomorphism and $\pi(\mathcal{A}\cap (\R^n\times \{t\}))=A_{\varphi,t}:=\{x\in A;\varphi(x)=t\}$.
For each $i\in \{1,...,k\}$, let $X_i$ be the closure of $B_i=\pi(\mathcal{B}_i)$ in $A$. Thus $A=\bigcup\limits_{i=1}^k X_i$ and each $X_i$ is a definable $C$-LLNE set w.r.t. $\varphi$. Note that each $X_i$ is also LNE at infinity. 

For each $t$, by Proposition \ref{prop:pancake_distance}, there exists a pancake distance of $A_{\varphi,t}:=\{x\in A;\varphi(x)=t\}= \bigcup\limits_{i=1}^k (X_i\cap A_{\varphi,t})$, denoted by $d_{A_{\varphi,t},P}$, which is definable in $\mathcal{S}$ and 
$$d_{A_{\varphi,t},P}(x,y)\leq d_{A_{\varphi,t},inn}(x,y)\leq C d_{A_{\varphi,t},P}(x,y),$$
for all $x,y\in A_{\varphi,t}$.

By the above discussion, it is enough to show that there are constants $K,r\geq 1$ such that for each $t\in (r,+\infty)$, we have
$$
d_{A_{\varphi,t},P}(x,y)\leq K d_{A,P} (x,y),
$$
for all $x,y\in A_{\varphi,t}$. 

The set $X=\phi(Y)$ is a definable set in $\mathcal{S}$, where $Y=\{(x,y,t)\in A\times A\times \R; \varphi(x)=\varphi(y)=t\}$ and $\phi\colon Y\to \R^2$ is given by $\phi(x,y, t)=(d_{A,P} (x,y),d_{A_{\varphi,t},P}(x,y))$.

Assume that Item (a) does not hold.
This implies that $(0,1)\in C(X,\infty)$, where $C(X,\infty)$ is the set of all the points $v\in \R^n$ such that there are sequences $\{ t_j \}_{j\in \N}\in (0,+\infty)$ and $\{x_j\}_{j\in \N}\subset X$ satisfying $\lim\limits_{j\to +\infty} t_j=+\infty$ and $\lim\limits_{j\to +\infty }\frac{1}{t_j}x_j=v.$

By following the proof of Proposition 2.15 in \cite{FernandesS:2020} and using the o-minimal versions of the Curve Selection Lemma (see \cite[Theorem 3.2]{Coste:1999}) and the Monotonicity Theorem (see \cite[Theorem 2.1]{Coste:1999}), we have the following characterization:
\begin{proposition}\label{prop:selection_lemma}
Let $Z\subset \R^n$ be an unbounded definable set in $\mathcal{S}$. A vector $w\in\R^n$ is a tangent vector of $Z$ at infinity if and only if there exists a continuous curve $\gamma\colon (\varepsilon ,+\infty )\to Z$, which is definable in $\mathcal{S}$,  such that $\lim\limits _{t\to +\infty }\|\gamma(t)\|=+\infty $ and $\gamma(t)=tw+o_{\infty }(t),$ where $g(t)=o_{\infty }(t)$ means $\lim\limits _{t\to +\infty }\frac{g(t)}{t}=0$. 
\end{proposition}

Thus, by Proposition \ref{prop:selection_lemma}, there exists a continuous arc $
\gamma\colon (\varepsilon,+\infty )\to A$ that is definable in $\mathcal{S}$ and such that $\lim\limits _{t\to +\infty }\|\gamma(t)\|=+\infty$ and $\gamma(t)=t(0,1)+o_{\infty }(t)$.
Let $
\gamma_1,\gamma_2\colon (\varepsilon,+\infty )\to A$ be continuous definable arcs such that $(\gamma_1(t),\gamma_2(t),t)\in Y$ and  $\phi(\gamma_1(t),\gamma_2(t),t)=\gamma(t)$ for all $t\in (\varepsilon,+\infty )$. Therefore,
$$
\lim\limits_{t\to +\infty}\frac{d_{A_{\varphi,t},P}(\gamma_1(t),\gamma_2(t))}{d_{A,P}(\gamma_1(t),\gamma_2(t))}=+\infty.
$$

By definition of the metric $d_{A,P}$, by taking a subsequence, if necessary, we may assume that there are natural numbers $r\in \mathbb{N}$ and $k_0,...,k_r\in \{1,...,k\}$, and a sequence $\{(x_0^j,x_1^j,...,x_r^j)\}_j\subset A^{r+1}$ such that for each $i\in\{0,...,r\}$  $x_i^j,x_{i+1}^j\in X_{k_i}$, where $x_0^j=\gamma_1(t_j)$, $x_{r+1}^j=\gamma_2(t_j)$ and $t_j=\varphi(x_0^j)=\varphi(x_{r+1}^j)$, for all $j$, and

$$
\frac{1}{C}\sum\limits_{i=0}^r \|x_i^j-x_{i+1}^j\|\leq d_{A,P}(x_0^j,x_{r+1}^j)\leq C\sum\limits_{i=0}^r \|x_i^j-x_{i+1}^j\|.
$$
By increasing $C$, if necessary, we may assume that, for each $i\in \{1,...,r\}$, $\|x_i^j\|\geq \frac{1}{1+C}\|x_0^j\|$, for all $j$. 
By Curve Selection Lemma (at infinity), we can choose a finite number of definable arcs $\tilde \beta_0, \tilde \beta_1, \ldots, \tilde \beta_{r+1}\colon (\epsilon, +\infty)\to A$ such that  $\lim\limits_{t\to +\infty}\|\tilde{\beta}_i(t)\|=+\infty$ for all $i\in \{1,\ldots,r\}$,
\begin{eqnarray*}
d(t)&:=&d_{A,P}(\gamma_1(t),\gamma_2(t))\\
    & \approx& \|\tilde \beta_0(t)-\tilde{\beta}_1(t)\|+\|\tilde \beta_1(t)-\tilde{\beta}_2(t)\|+\ldots+\|\tilde \beta_r(t)-\tilde \beta_{r+1}(t)\|=:\tilde d(t)
\end{eqnarray*}
and for each $i\in\{0,...,r\}$ the image of each pair $\tilde \beta_i$ and  $\tilde \beta_{i+1}$ is contained in some $X_{k_i}$, where $\tilde \beta_0=\gamma_1$ and $\tilde \beta_{r+1}=\gamma_2$.

For each $i\in\{0,1,\ldots,r+1\}$, let $\beta_i$ be the parametrization of $\tilde \beta_i$ such that $\varphi(\beta_i(t))=t$ for all large enough $t>0$ and let $h\colon (\epsilon,+\infty)\to \R$ be the function given by 
$$
h(t)=d_{A_t,P}(\gamma_1(t),\beta_1(t))+d_{A_t,P}(\beta_1(t),\beta_2(t))+ \ldots +d_{A_t,P}(\beta_r(t),\gamma_2(t)).
$$

Since $d_{A_{\varphi,t}, P}(\gamma_1(t),\gamma_2(t)) \leq h(t)$ for all $t$, then $d(t)\ll h(t)$. 

Let $C\geq 1$ be a constant such that $\frac{1}{C}\|x\|\leq\|\varphi(x)\|\leq C\|x\|$ for all $x\in A$. By increasing $C$, if necessary, we may assume that $X_i$ and $X_{i,\varphi,t}$ are $C$-LNE. Clearly, $d_{A_{\varphi,t}, P}(\gamma_1(t),\gamma_2(t)) \leq h(t)\leq 2(r+1)Ct$.  Then, for each $i$, $\|\tilde \beta_i(t)\|\approx t$.

\begin{proposition}[Isosceles property at infinity]\label{isosceles_property}
Let $\gamma_1,\gamma_2,\gamma_3\colon (r,+\infty)\to\mathbb{R}^n $ be arcs such that $\lim\limits_{t\to +\infty}\|\gamma_i(t)\|=+\infty$ for any $i\in\{1,2,3\}$. Assume that $\|\gamma_1-\gamma_2\|\lesssim \|\gamma_1-\gamma_3\|\lesssim \|\gamma_2-\gamma_3\|$. Then $\|\gamma_1-\gamma_3\|\approx \|\gamma_2-\gamma_3\|$.
\end{proposition}
\begin{proof}
We only have to show that $\|\gamma_2-\gamma_3\|\lesssim \|\gamma_1-\gamma_3\|$.
\begin{eqnarray*}
 \|\gamma_2-\gamma_3\|&\leq& \|\gamma_2-\gamma_1\|+\|\gamma_1-\gamma_3\|\\
            & \lesssim &  \|\gamma_1-\gamma_3\|+\|\gamma_1-\gamma_3\|\\
            & \lesssim & \|\gamma_1-\gamma_3\|.
\end{eqnarray*}
\end{proof}

Then, by Proposition \ref{isosceles_property},  $\|\beta_i(t)- \beta_{i+1}(t)\| \lesssim\|\tilde\beta_i(t)- \tilde\beta_{i+1}(t)\| $.
Therefore $h(t) \lesssim d(t)$, which is a contradiction with $d(t)\ll h(t)$.

Therefore Item (a) holds.

Similarly, we prove Item (b).
\end{proof}

\section{Some key results}\label{sec:key_results}
\subsection{Lipschitz normally embedding theorem}\label{subsec:lne_thm}

In this Subsection, we prove that any connected and definable set $X$ is definably inner lipeomorphic to a definable set that is LNE. This result was proved in \cite{BirbrairM:2000} in the case that $X$ was a compact and semi-algebraic set. We show here that the same proof of  \cite{BirbrairM:2000} works, with small changes, in our context.

\begin{definition}\label{general-lipschitz function}
Let $X\subset\R^n$ and $Y\subset\R^m$. Let $d_X$ and $d_Y$ be distances on $X$ and $Y$, respectively. A mapping $f\colon (X,d_X)\rightarrow (Y,d_Y)$ is called {\bf Lipschitz} if there exists $\lambda >0$ such that $d_{Y}(f(x_1),f(x_2))\le \lambda d_{X}(x_1,x_2))$  for all $x_1,x_2\in X$. In this case, $f$ is also called {\bf $\lambda$-Lipschitz}. A Lipschitz mapping $f\colon (X,d_X)\rightarrow (Y,d_Y)$ is called a {\bf lipeomorphism} if its inverse mapping exists and is Lipschitz  and, in this case, we say that  $(X,d_X)$ and $(Y,d_Y)$ are {\bf lipeomorphic}. We say that $f\colon (X,d_X)\rightarrow (Y,d_Y)$ is a {\bf lipeomorphism at infinity} if there are compact subsets $K$ and $\tilde K$ such that  $f\colon (X\setminus K,d_X|_{X\setminus K})\rightarrow (Y\setminus \tilde K,d_Y|_{Y\setminus \tilde K})$ is a lipeomorphism. In this case, we say that $(X,d_X)$ and $(Y,d_Y)$ are {\bf lipeomorphic at infinity}.
\end{definition}

If $d_X$ and $d_Y$ are the outer (resp. inner) distances, we add also the word ``outer'' (resp. ``inner'') in the above definitions.

When we say that two sets are {\bf definably lipeomorphic} means that there is a lipeomorphism between these two sets that is definable.

\begin{definition}[See \cite{BirbrairM:2000}]\label{def:lne}
Let $X\subset\R^n$ be a subset. We say that $X$ is {\bf Lipschitz normally embedded (LNE)} if there exists a constant $c\geq 1$ such that $d_{X,inn}(x_1,x_2)\leq C\|x_1-x_2\|$, for all pair of points $x_1,x_2\in X$.  We say that $X$ is {\bf Lipschitz normally embedded set at $p$} (shortly {\bf LNE at $p$}), if there is a neighbourhood $U$ such that $p\in U$ and $X\cap U$ is an LNE set or, equivalently, that the germ $(X,p)$ is LNE. In this case, we say also that $X$ is {\bf $C$-LNE} (resp. {\bf $C$-LNE at $p$}). We say that $X$ is {\bf Lipschitz normally embedded set at infinity} (shortly {\bf LNE at infinity}), if there is a compact subset $K$ such that $X\setminus K$ is an LNE set. In this case, we say also that $X$ is {\bf $C$-LNE at infinity}. 
\end{definition}

\begin{theorem}[Lipschitz normal embedding]
\label{thm:Lip_normal_embeddins}
Let $X\subset\R^n$ be a connected definable set in $\mathcal{S}$. Then there is a definable set $\widetilde{X}$ that is LNE and definably inner lipeomorphic to $X$.
\end{theorem}
\begin{proof}[Proof of Theorem \ref{thm:Lip_normal_embeddins}]
We set $\widetilde{X}^0=X$. Let $\{\widetilde{X}_i^0\}_{i=1}^k$ be a pancake decomposition of $\widetilde{X}^0$. Let $d_{\widetilde{X}^0,P}$ be the pancake distance given by the pancake decomposition $\{\widetilde{X}_i^0\}_{i=1}^k$.  

Assume that, for $j\geq 0$, $\widetilde{X}^j$ and $\{\widetilde{X}_i^j\}_{i=1}^k$, a pancake decomposition of $\widetilde{X}^j$, are defined. Let $d_{\widetilde{X}^j,P}$ be the pancake distance given by the pancake decomposition $\{\widetilde{X}_i^j\}_{i=1}^k$. We define $\mu_{j+1}\colon \widetilde{X}^j\to \R^{n+j+1}$ given by $\mu_{j+1}(x)=(x,h_{j+1}(x))$, where $h_{j+1}\colon \widetilde{X}^j\to \R$ is the function given by 
$$
h_{j+1}(x)=d_{\widetilde{X}^j,P}(x,X_{j+1}^j):=\inf \{d_{\widetilde{X}^j,P}(x,y);y\in X_{j+1}^j\}.
$$

Now, we set $\widetilde{X}^{j+1}=\mu_{j+1}(\widetilde{X}^j)$ and $\widetilde{X}_i^{j+1}=\mu_1(\widetilde{X}_i^j)$ for all $i\in \{1,...,k\}$. Note that $\{\widetilde{X}_i^{j+1}\}_{i=1}^k$ is a pancake decomposition of $\widetilde{X}^{j+1}$ and $\mu_{j+1}\colon \widetilde{X}^j\to \widetilde{X}^{j+1}$ is a definable inner lipeomorphism.

\begin{claim}\label{claim:rel_LNE}
There is a constant $K\geq 1$ such that $d_{\widetilde{X}^{j+1}, inn}(x,y)\leq K\|x-y\|$ for all $x\in \widetilde{X}_{j+1}^{j+1}$ and $y\in \widetilde{X}^{j+1}$.
\end{claim}
\begin{proof}
Since $\mu_{j+1}\colon \widetilde{X}^j\to \widetilde{X}^{j+1}$ is an inner lipeomorphism, it is enough to show that there is a constant $K\geq 1$ such that $d_{\widetilde{X}^{j}, P}(x,y)\leq K\|\mu_{j+1}(x)-\mu_{j+1}(y)\|$ for all $x\in \widetilde{X}_{j+1}^{j}$ and $y\in \widetilde{X}^{j}$.

Let $x\in \widetilde{X}_{j+1}^{j}$ and $y\in \widetilde{X}^{j}$. For $\epsilon>0$, consider $x_{\epsilon}\in \widetilde{X}^{j}$ such that $h_{j+1}(y)\geq d_{\widetilde{X}^{j}, P}(x_{\epsilon},y)-\epsilon$. By the definition of a pancake distance, we have
$$
d_{\widetilde{X}^{j}, P}(x,y)\leq \|x-x_{\epsilon}\|+d_{\widetilde{X}^{j}, P}(x_{\epsilon},y)
$$
Thus, if $h_{j+1}(y)\leq \|x-y\|$, we have that $\|x-x_{\epsilon}\|\leq 2\|x-y\|+\epsilon$.
Then, $d_{\widetilde{X}^{j}, P}(x,y)\leq 3\|x-y\|+\epsilon$, for all $\epsilon>0$, and thus $d_{\widetilde{X}^{j}, P}(x,y)\leq 3\|x-y\|$.

On the other hand, if $h_{j+1}(y)> \|x-y\|$, we have
\begin{eqnarray*}
d_{\widetilde{X}^{j}, P}(x,y)&\leq& d_{\widetilde{X}^{j}, P}(y,x_{\epsilon})+\|x-y\|+\|x_{\epsilon}-y\|\\
        &<&2d_{\widetilde{X}^{j}, P}(y,x_{\epsilon})+h_{j+1}(y)\\
        &<& 3h_{j+1}(y)+2\epsilon.
\end{eqnarray*}
Since $x\in \widetilde{X}_{j+1}^{j}$, we have that $h_{j+1}(x)=0$, and thus
$$
h_{j+1}(y)\leq \|\mu_{j+1}(x)-\mu_{j+1}(y)\|=\|(x-y,-h_{j+1}(y))\|.
$$
Then,
$$
d_{\widetilde{X}^{j}, P}(x,y)\leq 3\|\mu_{j+1}(x)-\mu_{j+1}(y)\|+\epsilon
$$
for all $\epsilon>0$, and thus $d_{\widetilde{X}^{j}, P}(x,y)\leq 3\|\mu_{j+1}(x)-\mu_{j+1}(y)\|$.

Therefore,
$d_{\widetilde{X}^{j}, P}(x,y)\leq 3\|\mu_{j+1}(x)-\mu_{j+1}(y)\|$ for all $x\in \widetilde{X}_{j+1}^{j}$ and $y\in \widetilde{X}^{j}$.
\end{proof}

Finally, we set $\widetilde{X}=\widetilde{X}^k$. 

By using the fact that, for each $j\in \{0,...,k-1\}$, $\mu_{j}\colon \widetilde{X}^{j}\to \widetilde{X}^{j+1}$ is a definable inner lipeomorphism, we have that $\widetilde{X}$ is definably inner lipeomorphic to $X$, and this together with Claim \ref{claim:rel_LNE}, we obtain that $\widetilde{X}$ is LNE.
\end{proof}

\subsection{Lipschitz arc-criterion}\label{subsec:arc-criterion_Lip}
In this Subsection, we present a criterion for a definable mapping to be Lipschitz when the involved distances are the inner or the outer distances.

\begin{proposition}
\label{prop:arc_lip_criterion}
Let $X\subset\R^n$ and $Y\subset\R^m$ be definable sets in $\mathcal{S}$ with connected links at infinity. Let $\sigma,\tilde\sigma\in \{inn,out\}$, $\phi\colon X\to \R$ be a radius function for $X$. Let $\varphi\colon (X,d_{X,\sigma})\to (Y,d_{Y,\tilde\sigma})$ be a definable mapping such that $\varphi|_{\Gamma}\colon (\Gamma,d_{X,\sigma}|_{\Gamma})\to (Y,d_{Y,\tilde\sigma})$ is Lipschitz at infinity for any definable curve $\Gamma\subset X$ with connected link at infinity. If $\varphi\colon (X,d_{X,\sigma})\to (Y,d_{Y,\tilde\sigma})$ is not Lipschitz at infinity, then there is a pair of definable arcs $\gamma_1,\gamma_2\colon (r, +\infty)\to X$ such that 
$d_{X,\sigma}(\gamma_1(t),\gamma_2(t))\ll d_{Y,\tilde\sigma}(\varphi(\gamma_1(t)),\varphi(\gamma_2(t)))$ and $\phi(\gamma_1(t))=\phi(\gamma_2(t))=t$ for all big enough $t> 0$.
\end{proposition}
\begin{proof}

If $\sigma=out$ (resp. $\tilde\sigma=out$), we set $\widetilde{X}=X$ (resp. $\widetilde{Y}=Y$) and $\mu_1\colon X\to \widetilde{X}$ (resp. $\mu_2\colon X\to \widetilde{X}$) is the identity mapping, and if $\sigma=inn$ (resp. $\tilde\sigma=inn$), then $\widetilde{X}$ (resp. $\widetilde{Y}$) the Lipschitz normal embedding of $X$ (resp. $Y$) and $\mu_1\colon X\to \widetilde{X}$ (resp. $\mu_2\colon X\to \widetilde{X}$) is the definable inner lipeomorphism given by Theorem \ref{thm:Lip_normal_embeddins}.

Let $\tilde\varphi\colon \widetilde{X} \to \widetilde{Y}$ be the mapping defined as $\tilde\varphi=\mu_2^{-1}\circ \varphi \circ \mu_1$. Thus, $\varphi$ is Lipschitz if and only if $\tilde\varphi$ is outer Lipschitz.

Let $\tilde\phi\colon \widetilde{X}\to \R$ be the function given by $\tilde\phi=\phi\circ \mu_1^{-1}$. Note that $\mu_1^{-1}$ is a restriction of a linear projection, and therefore it is an outer Lipschitz mapping. Thus, $\tilde\phi $ is a radius function for $\widetilde{X}$. Therefore, the following two items are equivalent:
\begin{enumerate}
 \item There is a pair of definable arcs $\gamma_1,\gamma_2\colon (r, +\infty)\to X$ such that $d_{X,\sigma}(\gamma_1(t),\gamma_2(t))\ll d_{Y,\tilde\sigma}(\varphi(\gamma_1(t)),\varphi(\gamma_2(t)))$ and $\phi(\gamma_1(t))=\phi(\gamma_2(t))=t$ for all big enough $t> 0$;
 \item There is a pair of definable arcs $\tilde\gamma_1,\tilde\gamma_2\colon (r, +\infty)\to \widetilde{X}$ such that $\|\tilde\gamma_1(t)-\tilde\gamma_2(t)\|\ll \| \tilde\varphi(\tilde\gamma_1(t))-\tilde\varphi(\tilde\gamma_2(t))\|$ and $\tilde\phi(\tilde\gamma_1(t))=\tilde\phi(\tilde\gamma_2(t))=t$ for all big enough $t> 0$.
\end{enumerate}
This is why, we may assume that $\sigma=\tilde\sigma=out$.

Let $A=\{(u,v,\epsilon)\in \iota(X\setminus\{0\})\times \iota(X\setminus\{0\})\times (0,+\infty);x\not=y,\, \|x\|, \|y\|\leq \epsilon$ and $\|\iota(x)-\iota(y)\|<\epsilon \|\varphi(\iota(x))-\varphi(\iota(y))\|\}$, where $\iota\colon \R^n\setminus \{0\}\to \R^n\setminus \{0\}$ is the mapping given by $\iota(x)=\frac{x}{\|x\|^2}$.
We have that $A$ is a definable set and if $\varphi$ is not Lipschitz at infinity, then $0\in \overline{A}$. By the Curve Selection Lemma, there is a definable arc $\beta=(\beta_1,\beta_2,\beta_3)\colon [0,\varepsilon)\to \overline{A}$ such that $\beta(0)=0$ and $\beta(t)\in A$ for all $t\in (0,\varepsilon)$. Let $\alpha_i\colon(r,+\infty)\to X$ given by $\alpha_i(t)=\iota(\beta_i(\frac{1}{t}))$, $i=1,2$ with $r=\frac{1}{\varepsilon}$.  Note that 
$$
\|\alpha_1(t)-\alpha_2(t)\|\ll \|\varphi(\alpha_1(t))-\varphi(\alpha_2(t))\|.
$$

Now, we divide our proof in the following two cases.

\noindent Case 1. The radius function $\phi$ satisfies: $\phi(x)=\|x\|$ for all $x\in X$.

By reordering the indices, if necessary, we may assume that $\|\alpha_1(t)\|\geq \|\alpha_2(t)\|$ for all big enough $t$. After a reparametrization, we may assume that $\|\alpha_2(t)\|=t$ for all big enough $t$. Let $\gamma_1$ be the parametrization of $\alpha_1$ such that $\gamma_1\colon (r',+\infty)\to Im(\gamma_1)\subset X$ is an outer lipeomorphism and $\|\gamma_1(t)\|=t$ for all big enough $t>0$. Then, $\|\gamma_1(t)-\gamma_2(t)\|\lesssim \|\alpha_1(t)-\gamma_2(t)\|$, where $\gamma_2=\alpha_2$. Indeed, since  $\gamma_1\colon (r',+\infty)\to Im(\gamma_1)$ is an outer lipeomorphism and $\alpha_1(t)=\gamma_1(\|\alpha_1(t)\|)$ for all big enough $t$, we have that 
$$
\|\alpha_1(t)-\gamma_1(t)\|=\|\gamma_1(\|\alpha_1(t)\|)-\gamma_1(t)\|\approx |\|\alpha_1(t)\|-t|.
$$
But the triangle with vertices $\alpha_1(t)$, $\alpha_1'(t):=t\frac{\alpha_1(t)}{\|\alpha_1(t)\|}$ and $\gamma_2(t)$ is obtuse (see Figure \ref{fig:angle_one}), then we have that 
$$
|\|\alpha_1(t)\|-t|=|\|\alpha_1(t)\|-\|\gamma_2(t)\||\lesssim \|\alpha_1(t)-\gamma_2(t)\|.
$$ 
Then, by the isosceles property at infinity (see Proposition \ref{isosceles_property}), we have
$$
\|\gamma_1(t)-\gamma_2(t)\|\lesssim \|\alpha_1(t)-\gamma_2(t)\|.
$$

\begin{figure}[h]
    \centering
\includegraphics[height=8cm]{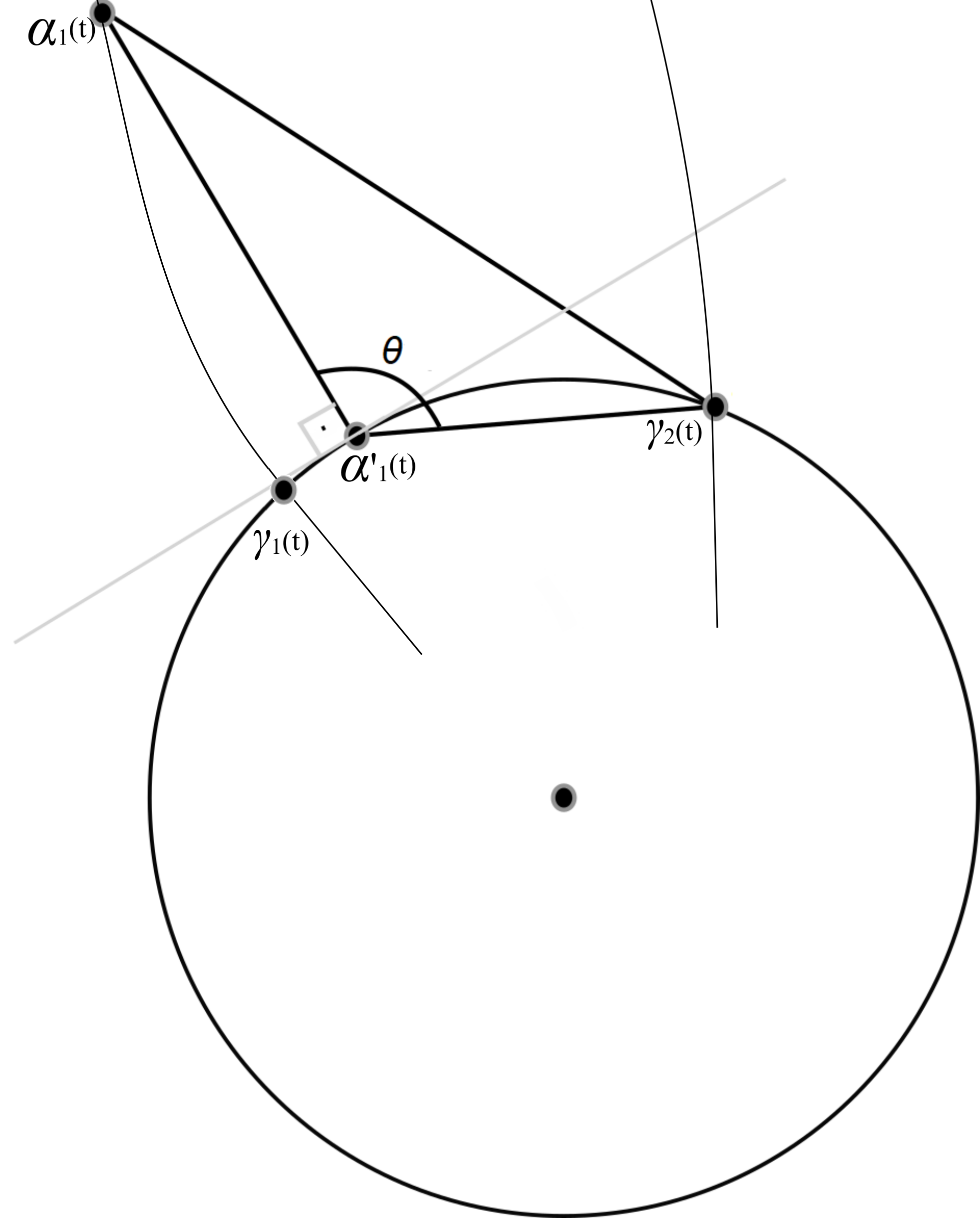}
\caption{Obtuse triangle.}
\label{fig:angle_one}
\end{figure}

Moreover, by the assumptions, we have 
$$
\|\varphi(\alpha_1(t))-\varphi(\gamma_1(t))\|\lesssim \|\alpha_1(t)-\gamma_1(t)\|
$$
for all big enough $t>0$. 
We assume by contradiction that $\|\varphi(\gamma_1(t))-\varphi(\gamma_2(t))\| \lesssim \|\gamma_1(t)-\gamma_2(t)\|$.
Then
\begin{eqnarray*}
\|\alpha_1(t)-\gamma_2(t)\|&\ll& \|\varphi(\alpha_1(t))-\varphi(\gamma_2(t))\|\\
                &\lesssim& \|\varphi(\alpha_1(t))-\varphi(\gamma_1(t))\|+ \|\varphi(\gamma_1(t))-\varphi(\gamma_2(t))\|\\
                &\lesssim& \|\alpha_1(t)-\gamma_1(t)\|+ \|\gamma_1(t)-\gamma_2(t)\|.            
\end{eqnarray*}

This implies that $\|\gamma_1(t)-\gamma_2(t)\|\lesssim \|\alpha_1(t)-\gamma_2(t)\|\ll \|\alpha_1(t)-\gamma_1(t)\|$, which is a contradiction with the isosceles property at infinity (see Proposition \ref{isosceles_property}).
Therefore, $\|\gamma_1(t)-\gamma_2(t)\|\ll \|\varphi(\gamma_1(t))-\varphi(\gamma_2(t))\|$.
In fact, we have that $\|\gamma_1(h(t))-\gamma_2(h(t))\|\ll \|\varphi(\gamma_1(h(t)))-\varphi(\gamma_2(h(t)))\|$, for any definable function $h\colon (r',+\infty)\to (r,\infty)$ such that $\lim\limits_{t\to +\infty}h(t)=+\infty$.

\noindent Case 2. We have a general radius function $\phi$.
In this case, we note that $\phi$ is the restriction of the radius function for $\R^n$, $\tilde\phi\colon \R^n\to \R$, given by $\tilde\phi(y)=\min \{\max\{\frac{1}{C}\|y\|, \bar{\phi}(y)\}, C\|x\|\}$, where $C\geq 1$ is a constant such that $\frac{1}{C}\|x\|\leq \phi(x)\leq C\|x\|$, for all $x\in X$ and $\bar{\phi}\colon \R^n\to \R$ is the definable outer Lipschitz function given by $\bar{\phi}(y)=\inf\{\phi(x)+C\|x-y\|; x\in X\}$. 
By increasing $C$, if necessary, we may assume that $\tilde \phi$ is outer $C$-Lipschitz.

We define $\psi\colon \R^n\to \R^n$ by
$$
\psi(x)=\left\{\begin{array}{ll}
\frac{\tilde\phi(x)}{\|x\|}x,&\mbox{ if }x\not=0,\\
0,&\mbox{ if }x\not=0.
\end{array}\right.
$$

Let $\widetilde \psi \colon \R^n\to \R^n$ given by $\widetilde\psi(x)=\iota\circ\psi\circ\iota(x)$ for $x\not=0$ and $\tilde\psi(0)$. Note that 
$$
\widetilde\psi(x)=\left\{\begin{array}{ll}
\frac{\widetilde{\tilde\phi}(x)}{\|x\|}x,&\mbox{ if }x\not=0,\\
0,&\mbox{ if }x\not=0,
\end{array}\right.
$$
where $\widetilde{\tilde\phi}\colon \R^n\to \R$ is the radius function given by
$$
\widetilde{\tilde\phi}(x)=\left\{\begin{array}{ll}
\frac{1}{\tilde\phi\circ\iota(x)},&\mbox{ if }x\not=0,\\
0,&\mbox{ if }x\not=0.
\end{array}\right.
$$

Thus, for $x,y\in \R^n\setminus\{0\}$, we have
\begin{eqnarray*}
\|\widetilde\psi(x)-\widetilde\psi(y)\|&\leq& \left\|\frac{\widetilde{\tilde\phi}(x)}{\|x\|}x-\frac{\widetilde{\tilde\phi}(y)}{\|y\|}y\right\|\\
                   &\leq& \frac{\widetilde{\tilde\phi}(x)}{\|x\|}\|x-y\|+\frac{1}{\|x\|}\cdot \left|\widetilde{\tilde\phi}(x)\|y\|-\widetilde{\tilde\phi}(y)\|x\|\right|\\
                   &\leq & C\|x-y\|+\frac{\widetilde{\tilde\phi}(x)}{\|x\|}\cdot \left|\|y\|-\|x\|\right|+ \left|\widetilde{\tilde\phi}(x)-\widetilde{\tilde\phi}(y)\right|\\
                   &\leq & 3C\|x-y\|
\end{eqnarray*}

In the same way, we prove also that $\|\psi(x)-\psi(y)\|\leq  3C\|x-y\|$. 
Therefore, $\widetilde\psi$ and $\psi$ are outer Lipschitz mappings. In fact, $\widetilde\psi\colon (\R^n,0)\to (\R^n,0)$ is an outer lipeomorphism (see \cite[Lemma 2.8]{Nhan:2023}). 

Let $K$ be a constant such that $\widetilde\psi^{-1}$ is outer $K$-Lipschitz in some ball $B_r^n(0)$. Thus,  by using Rademacher's theorem, the derivative $D\widetilde\psi^{-1}_x$ exists almost everywhere on $B_r^n(0)$ and $\|D\widetilde\psi^{-1}_x\|\leq K$. Then, $D\psi^{-1}(x)$ exists almost everywhere on $\R^n\setminus B_{\frac{1}{r}}(0)$ and 
\begin{eqnarray*}
\|D\psi^{-1}_y\|&=&\|D\iota_{\widetilde\psi^{-1}(\iota(y))}\cdot D\widetilde\psi^{-1}_{\iota(y)}\cdot D\iota_y\| \\
                &\leq& \|D\iota_{\widetilde\psi^{-1}(\iota(y))}\|\cdot\| D\widetilde\psi^{-1}_{\iota(y)}\|\cdot\| D\iota_y\|\\
                &\leq &\frac{1}{\|\widetilde\psi^{-1}(\iota(y))\|^2}\cdot K \cdot \frac{1}{\|y\|^2}\\
                &\leq& 9C^2K.
\end{eqnarray*}
Therefore, $\psi^{-1}\colon \R^n\setminus B_{\frac{1}{r}}(0)\to \R^n$ is an outer Lipschitz mapping. Since $\psi\colon \R^n\to \R^n$ is an outer Lipschitz mapping, then $\psi$ is a definable outer lipeomorphism at infinity (see also Theorem \ref{thm:charac_outer_geo_lip}).

Let $X^{\psi}=\psi^{-1}(X)$. By the Case 1 of this proof, 
there is a pair of definable arcs $\tilde\gamma_1,\tilde\gamma_2\colon (r, +\infty)\to X^{\psi}$ such that 
$$
\|\tilde\gamma_1(t)-\tilde\gamma_2(t)\|\ll \| \varphi\circ \psi(\tilde\gamma_1(t))-\varphi\circ \psi(\tilde\gamma_2(t))\|
$$
and $\|\tilde\gamma_1(t)\|=\|\tilde\gamma_2(t)\|=t$ for all big enough $t> 0$.

Let $\gamma_i=\psi\circ \tilde\gamma_2(t)$ $i=1,2$. Therefore,
$$
\|\gamma_1(t)-\gamma_2(t)\|\ll \|\varphi(\gamma_1(t))-\varphi(\gamma_2(t))\|
$$
and $\phi(\gamma_1(t))=\phi(\gamma_2(t))=t$, which finishes the proof.
\end{proof}

Similarly, we obtain also the following result:
\begin{proposition}
\label{prop:local_arc_lip_criterion}
Let $X\subset\R^n$ and $Y\subset\R^m$ be definable sets in $\mathcal{S}$ with connected links at 0. Let $\sigma,\tilde\sigma\in \{inn,out\}$, $\phi\colon X\to \R$ be a radius function for $X$. Let $\varphi\colon (X,d_{X,\sigma})\to (Y,d_{Y,\tilde\sigma})$ be a definable mapping such that $\varphi|_{\Gamma}\colon (\Gamma,d_{X,\sigma}|_{\Gamma})\to (Y,d_{Y,\tilde\sigma})$ is Lipschitz around $0$ for any definable curve $\Gamma\subset X$ with connected link at $0$. If $\varphi\colon (X,d_{X,\sigma})\to (Y,d_{Y,\tilde\sigma})$ is not Lipschitz around $0$, then there is a pair of definable arcs $\gamma_1,\gamma_2\colon (0, \varepsilon)\to X$ such that 
$d_{X,\sigma}(\gamma_1(t),\gamma_2(t))\ll d_{Y,\tilde\sigma}(\varphi(\gamma_1(t)),\varphi(\gamma_2(t)))$ and $\phi(\gamma_1(t))=\phi(\gamma_2(t))=t$ for all small enough $t> 0$.
\end{proposition}

Thus, we obtain the following LNE at infinity arc-criterion:

\begin{corollary}\label{prop:arc-criterion}
Let $A\subset \R^n$ be a definable set in $\mathcal{S}$ with connected link at infinity. Let $\phi\colon A\to \R$ be a radius function for $A$, Then $A$ is not LNE at infinity if and only if there is a pair of definable arcs in $\gamma_1,\gamma_2\colon (\varepsilon, +\infty)\to A$ such that $
\|\gamma_1(t)-\gamma_2(t)\|\ll d_{A,inn}(\gamma_1(t),\gamma_2(t))$ and $\phi(\gamma_1(t))=\phi(\gamma_2(t))=t$ for all big enough $t> 0$. 
\end{corollary}
\begin{proof}

It is clear that if there is a pair of arcs as above, then $A$ is not LNE at infinity.

Reciprocally, assume that $A$ is not LNE at infinity.
Thus, the identity mapping $\varphi=id\colon (A,d_{A,out})\to (A,d_{A,inn})$ is a definable mapping that is not Lipschitz. However, $\varphi|_{\Gamma}\colon (\Gamma,d_{A,out}|_{\Gamma})\to (A,d_{A,inn})$ is Lipschitz at infinity for any definable curve $\Gamma\subset A$ with connected link at infinity. Thus, the result follows from Proposition \ref{prop:arc_lip_criterion}.
\end{proof}

Similarly, we obtain also the following local LNE arc-criterion, which generalizes the main result in \cite{BirbrairM:2018}.

\begin{corollary}\label{prop:local-arc-criterion}
Let $A\subset \R^n$ be a connected definable set in $\mathcal{S}$. Let $\phi\colon A\to \R$ be a radius function for $A$, Then $A$ is not LNE at infinity if and only if there is a pair of definable arcs in $\gamma_1,\gamma_2\colon [0, \epsilon)\to A$ such that $
\|\gamma_1(t)-\gamma_2(t)\|\ll d_{A,inn}(\gamma_1(t),\gamma_2(t))$ and $\phi(\gamma_1(t))=\phi(\gamma_2(t))=t$ for all small enough $t> 0$. 
\end{corollary}

\section{LNE vs. LLNE}\label{sec:char_lne}

\begin{definition}
Let $X\subset\R^n$ be a subset, $p\in \overline X$ and $X_t:=X\cap \mathbb{S}_t^{n-1}(p)$ for all $t>0$. We say that $X$ is {\bf link Lipschitz normally embedded at $p$} (or shortly {\bf LLNE at $p$}), if there is a constant $C\geq 1$ such that $d_{X_t, inn}\leq C\|\cdot \|$, for all small enough $t>0$. In this case, we say also that $X$ is $C$-LLNE at $p$. We say that $X$ is {\bf link Lipschitz normally embedded at infinity}  (or shortly {\bf LLNE at infinity}), if there is a constant $C\geq 1$ such that $d_{X_t, inn}\leq C\|\cdot \|$, for all large enough $t>0$. In this case, we say also that $X$ is {\bf $C$-LLNE at infinity}. 
\end{definition}

\begin{definition}
Let $X\subset\R^n$ be a subset. Let $\varphi\colon A\to \R$ be a radius function and $X_{\varphi,t}:=\{x\in X;\varphi(x)=t\}$ for all $t>0$. We say that $X$ is {\bf link LNE at $0$} (resp. {\bf infinity}) {\bf with respect to $\varphi$}, if there is a constant $C\geq 1$ such that $d_{X_{\varphi,t}, inn}\leq C\|\cdot \|$, for all small (resp. large) enough $t>0$. In this case, we say also that $X$ is {\bf $C$-LLNE at $0$} (resp.  {\bf infinity}) {\bf w.r.t. $\varphi$}. 
\end{definition}

As direct consequences of Theorem \ref{main_theorem}, we obtain the following relations between LNE and LLNE notions:
\begin{corollary}\label{cor:lne_llne_equiv}
Let $A\subset \R^n$ be a definable set in $\mathcal{S}$ with connected link at infinity. Then we have the following equivalent statements:
\begin{enumerate}
 \item [(1)] $A$ is LNE at infinity;
 \item [(2)] $A$ is LLNE at infinity;
 \item [(3)] $A$ is LLNE at infinity w.r.t. any radius function $\varphi\colon A\to \R$.
\end{enumerate}
\end{corollary}
\begin{proof}
$(3)\Rightarrow (2)$ is trivial.

$(2)\Rightarrow (1)$. 
Suppose that $A$ is not LNE at infinity. 

By Corollary \ref{prop:arc-criterion}, there is a pair of definable arcs in $\gamma_1,\gamma_2\colon (\varepsilon, +\infty)\to A$ such that $
\|\gamma_1(t)-\gamma_2(t)\|\ll d_{A,inn}(\gamma_1(t),\gamma_2(t))$ and $\|\gamma_1(t)\|=\|\gamma_2(t)\|=t$ for all big enough $t> 0$. 
Since $d_{A,inn}(\gamma_1(t),\gamma_2(t))\leq d_{A_t, inn}(\gamma_1(t),\gamma_2(t))$, for all big enough $t>0$, it follows that $\|\gamma_1(t)-\gamma_2(t)\|\ll d_{A_t,inn}(\gamma_1(t),\gamma_2(t))$, where $A_t=\{x\in A;\|x\|=t\}$. So, $A$ is not LLNE at infinity.

$(1)\Rightarrow (3)$. Assume that $A$ is LNE at infinity. So, there are a compact subset $\tilde K\subset \R^n$ and a constant $C\geq 1$ such that $d_{A,inn}(x,y)\leq C \|x-y\|$, for all $x,y\in A\setminus \tilde K$. 

Let $\varphi\colon A\to \R$ be a radius function.
By Theorem \ref{main_theorem}, there are constant $r,K\geq 1$ such that $d_{A_{\varphi,t},inn}(x,y)\leq K d_{A,inn}(x,y)$ for all $x,y\in A_{\varphi,t}$ and $t>r$. Then,
$$
d_{A_{\varphi,t},inn}(x,y)\leq KC\|x-y\|
$$
for all $x,y\in A_{\varphi,t}$ and all big enough $t>0$. Therefore $A$ is LLNE w.r.t. $\varphi$.
\end{proof}

Similarly, we also have the local version of the above result, which is an adaptation of the main result in \cite{mendessampaio} and was already proved in \cite{Nhan:2023}.

Given a set $X\subset\R^n$, the {\bf inversion of $X$} is the set $\iota(X\setminus \{0\})$, where $\iota\colon \R^n\setminus \{0\}\to \R^n\setminus \{0\}$ is the mapping given by $\iota(x)=\frac{x}{\|x\|^2}$. The mapping $\iota$ is called the {\bf inversion mapping of $\R^n$}.

Let $\rho\colon \mathbb{S}^n\setminus \{e_{n+1}\}\to\R^n$ be the stereographic projection (of the $e_{n+1}$), which is given by $\rho(x,t)=\frac{x}{1-t}$, where $e_{n+1}=(0,...,0,1)$.
We denote by $\widehat{X}$ the set $\rho^{-1}(X)\cup \{e_{n+1}\}$ whenever $X$ is an unbounded subset of $\R^n$. When $X$ is a bounded subset of $\R^n$, $\widehat{X}$ is just $\rho^{-1}(X)$. Note that when $X$ is a closed subset of $\R^n$, $\widehat{X}$ is a one-point compactification of $X$. $\widehat{X}$ is called the {\bf stereographic modification of $X$}.

Thus, by Corollaries \ref{prop:arc-criterion}, \ref{prop:local-arc-criterion} and  \ref{cor:lne_llne_equiv} (and its local version), we obtain the following:
\begin{corollary}
\label{cor:charac_LNE_infinity}
Let $X\subset\R^n$ be a definable set in $\mathcal{S}$ with connected link at infinity. Then, the following statements are equivalent:
\begin{enumerate}
 \item $X$ is LNE at infinity;
 \item Its stereographic modification $\widehat{X}$ is LNE at $e_{n+1}$;
 \item Its inversion $\iota(X\setminus \{0\})$ is LNE at $0$.
\end{enumerate}
\end{corollary}

Since for a definable set $X$, the identity mapping of $X$ is a definable mapping that preserves the outer distance to the origin and also preserves the last coordinate, then Corollary \ref{cor:charac_LNE_infinity} is also a direct consequence of Corollary \ref{thm:charac_geo_lip}.

Since the stereographic projection is an outer lipeomorphism far from $e_{n+1}$ and the inversion mapping is an outer lipeomorphism far from $e_{n+1}$ and infinity, we obtain also the following results.

\begin{corollary}
\label{cor:charac_LNE_infinity_two}
Let $X\subset\R^n$ be a connected definable set in $\mathcal{S}$.  Then, the following statements are equivalent:
\begin{enumerate}
 \item $X$ is LNE;
 \item $\widehat{X}$ is LNE;
\end{enumerate}
Moreover, if $X\setminus \{0\}$ is a connected set, then (1) and (2) are equivalent to the following:
\begin{enumerate}
 \item [(3)] $\iota(X\setminus\{0\})$ is LNE.
\end{enumerate}
\end{corollary}

%

We obtain also a positive answer to Question \ref{question:dependency_centre}.
 
\begin{corollary}\label{cor:dependency_centre}
Let $X\subset \R^n$ be a definable set in $\mathcal{S}$. Then, there is a constant $C\geq 1$ such that $X\cap \mathbb{S}_t^{n-1}(p)$ is $C$-LNE, for all large enough $t>0$.
\end{corollary}
\begin{proof}
Indeed, $\varphi \colon Y:=X\setminus B_{R}^n(0)\to \R$ given by $\varphi(x)=\|x-p\|$ is a radius function for $Y$, and $Y$ is LLNE at infinity if and only if $X$ is. By Corollary \ref{cor:lne_llne_equiv}, $Y$ is LLNE at infinity if and only if $Y$ is LLNE at infinity w.r.t. $\varphi$. However, $Y\cap \mathbb{S}_t^{n-1}(p)=Y_{\varphi,t}$. Thus if $X\subset \R^n$ is LLNE at infinity and $p\in \R^n\setminus\{0\}$, there is a constant $C\geq 1$ such that $X\cap \mathbb{S}_t^{n-1}(p)$ is $C$-LNE, for all large enough $t>0$.
\end{proof}

\section{Outer Lipschitz geometry: local vs. global}
\label{sec:outer_geo}

\begin{theorem}\label{thm:charac_outer_geo_lip}
Let $X\subset\R^n$ and $Y\subset\R^m$ be sets.
Then, the following statements are equivalent:
\begin{enumerate}
 \item $X$ and $Y$ are outer lipeomorphic at infinity;
 \item The germs of the stereographic modifications $(\widehat{X},e_{n+1})$ and $(\widehat{Y},e_{m+1})$ are outer lipeomorphic;
 \item The germs of the inversions $(\iota(X\setminus \{0\}),0)$ and $(\iota(Y\setminus \{0\}),0)$ are outer lipeomorphic. 
\end{enumerate}
\end{theorem}

After this article was finished, the author realized that Theorem \ref{thm:charac_outer_geo_lip}, which is stated below, was already proved in the preprint \cite{GrandjeanO:2023}, which appeared on arXiv fill days before this article. The reader should compare the proofs.
The main difference is that here we give a direct proof instead of a proof by contradiction as it was done in \cite{GrandjeanO:2023}. 

\begin{proof}[Proof of Theorem \ref{thm:charac_outer_geo_lip}]

Let $A,B\subset \R^N$ be subsets and let
$\varphi \colon \R^N\to \R^N$ be an outer lipeomorphism such that $\varphi(A)=B$ and $\varphi(0)=0$.

Let $\tilde \varphi\colon \R^N\setminus\{0\}\to \R^N\setminus\{0\}$ be the mapping $\tilde\varphi=\iota\circ \varphi\circ \iota$. Clearly, $\tilde \varphi$ is a homeomorphism such that $\tilde \varphi(\iota(A\setminus\{0\}))=\iota(B\setminus\{0\})$.

\begin{claim}\label{claim:key_claim}
 $\tilde \varphi$ is an outer lipeomorphism.
\end{claim}
\begin{proof}
This claim follows from Rademacher's theorem as was already done in the proof of Proposition \ref{prop:arc_lip_criterion}, but we present here a more geometric proof.

Let us analyse $\tilde\varphi$ by assuming that $\varphi$ and $\varphi^{-1}$ are outer $C-$Lipschitz.
When $\|x\|=\|y\|$, we have the following:
\begin{eqnarray*}
\|\tilde\varphi(x)-\tilde\varphi(y)\|&=&\left\|\frac{\varphi(\iota(x))}{\|\varphi(\iota(x))\|^2}-\frac{\varphi(\iota(y))}{\|\varphi(\iota(y))\|^2} \right\|\\
 &\leq & \left(\frac{2C}{\|\varphi(\iota(x))\|^2}+\frac{2C}{\|\varphi(\iota(y))\|^2}\right)\|\iota(x)-\iota(y)\|\\
  &\leq & 4C^3\|x-y\|.
\end{eqnarray*}
When $y=\lambda x$ and $\lambda \geq 1$, we have the following:
\begin{eqnarray*}
\|\tilde\varphi(x)-\tilde\varphi(y)\|&=&\left\|\frac{\varphi(\iota(x))}{\|\varphi(\iota(x))\|^2}-\frac{\varphi(\iota(y))}{\|\varphi(\iota(y))\|^2} \right\|\\
 &\leq & \frac{1}{\|\varphi(\iota(x))\|^2}\left\|\varphi(\iota(x))-\varphi(\iota(y)) \right\|+\\
 & & + \frac{\|\varphi(\iota(x))\|+\|\varphi(\iota(y))\|}{\|\varphi(\iota(x))\|^2\|\varphi(\iota(y))\|}\left\|\varphi(\iota(x))-\varphi(\iota(y)) \right\|\\
  &\leq & C\|x\|^2\left\|\frac{x}{\|x\|^2}-\frac{x}{\lambda\|x\|^2} \right\|+C^4\|x\|^2\frac{\lambda+1}{\lambda}\left\|\frac{x}{\|x\|^2}-\frac{x}{\lambda\|x\|^2} \right\|\\
  &= & \left(\frac{C}{\lambda}+C^4\frac{\lambda+1}{\lambda}\right)\|x-y\|\\
  &= & \left(C+2C^4\right)\|x-y\|.
\end{eqnarray*}

Thus, for $u,v\in \R^N\setminus \{0\}$ with $\|u\|\leq \|v\|$, let $v'=\frac{\|u\|}{\|v\|}v$. Thus, we have $\|u-v\|\geq \|u-v'\|$ and $\|u-v\|\geq \|v-v'\|$ (see Figure \ref{fig:angle}, which is a slight modification of Fig. 2 in \cite{FernandesS:2019}).
\begin{figure}[h]
    \centering
\includegraphics[height=8cm]{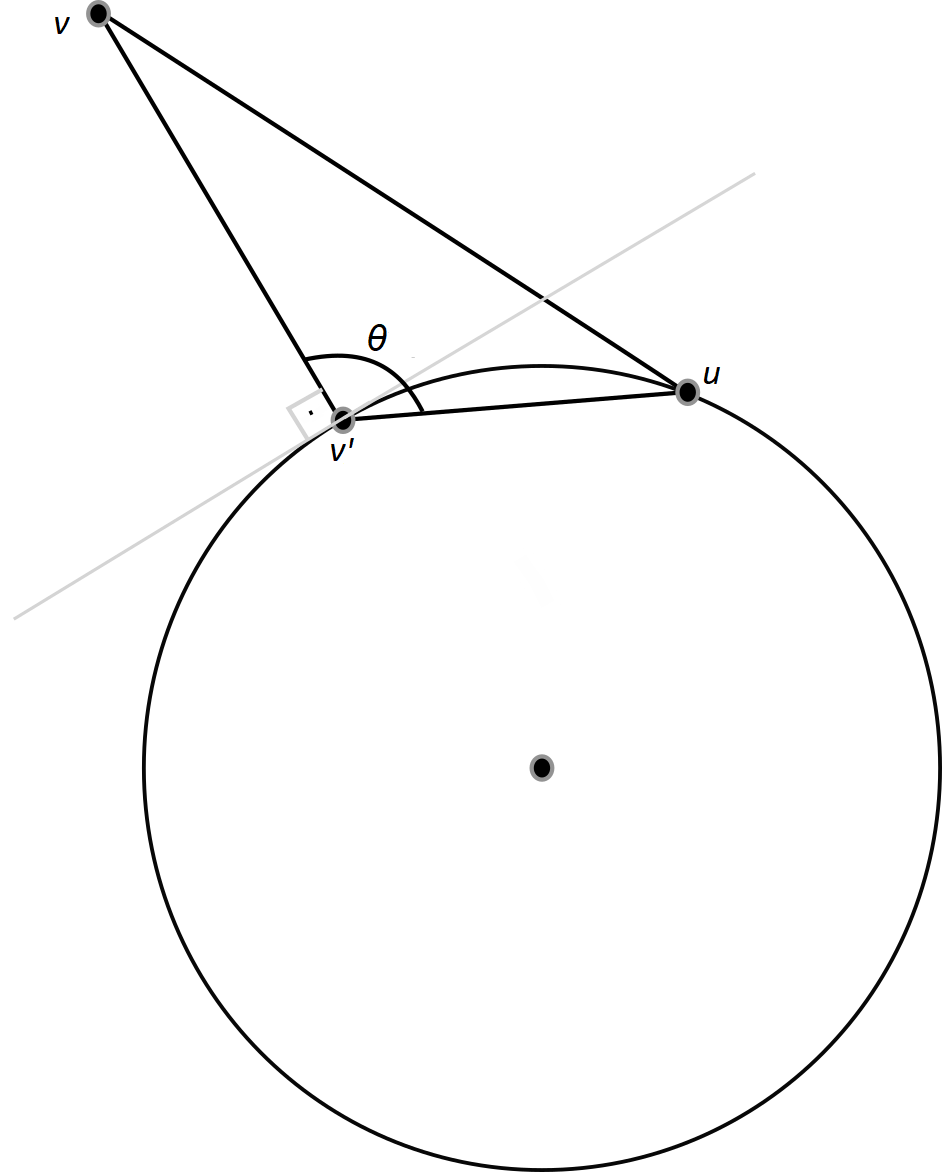}
\caption{Obtuse angle ($\theta>\frac{\pi}{2}$).}
\label{fig:angle}
\end{figure}

Therefore
\begin{eqnarray*}
\|\tilde\varphi(u)-\tilde\varphi(v)\|&\leq&\|\tilde\varphi(u)-\tilde\varphi(v')\|+\|\tilde\varphi(v')-\tilde\varphi(v)\|\\
 &\leq &  4C^3\|u-v'\|+\left(C+2C^4\right)\|v'-v\|\\
 &\leq &  4C^3\|u-v\|+\left(C+2C^4\right)\|u-v\|\\
  &= & (4C^3+C+2C^4)\|u-v\|.
\end{eqnarray*}

Therefore $\tilde \varphi$ is outer Lipschitz.
Similarly, we prove that $\tilde \varphi^{-1}=\iota\circ\varphi^{-1} \circ \iota$ is outer Lipschitz as well.
 
\end{proof}

Now, if we have an outer lipeomorphism $\varphi\colon A\subset \R^n\to B\subset \R^m$. By changing $A$ and $B$ by $ \{0\}\times A\subset \R^m\times \R^n$ and $ \{0\}\times B\subset \R^n\times \R^m$, respectively, we may assume that $n=m=N$ and $\varphi$ is an outer lipeomorphism $\varphi \colon \R^N\to \R^N$ such that $\varphi(A)=B$ (see the details in \cite{Sampaio:2016}, \cite{FernandesS:2020} and \cite{FernandesS:2023}). Since we are interested only in two cases: around $0$ or far from $0$, we may assume that $\varphi(0)=0$. Thus, by Claim \ref{claim:key_claim}, we obtain the equivalence $(1) \Leftrightarrow (3)$.

Now, we are going to show that $(1)\Leftrightarrow (2)$. 
By considering the identifications $\R^n\cong \{0\}\times \R^n$ and $\R^m\cong \{0\}\times \R^m$, we may assume that $m=n$.

Remind that $\iota\colon \R^n\setminus \{0\}\to \R^n\setminus \{0\}$ is the mapping given by $\iota(x)=\frac{x}{\|x\|^2}$ and $\rho\colon \mathbb{S}^n\setminus \{e_{n+1}\}\to\R^n$ is given by $\rho(x,t)=\frac{x}{1-t}$. Let $\widehat{\rho}\colon \mathbb{S}^n\setminus \{-e_{n+1}\}\to\R^n$ be the stereographic projection of the $-e_{n+1}$, which is given by $\widehat\rho(x,t)=\frac{x}{1+t}$. Then $\iota\circ \widehat{\rho}=\rho$.

Now, for any $-1<r<1$, $\widehat{\rho}|_A\colon A:=\{(x,t)\in\mathbb{S}^n;t> r\}\to B_R^n(0)$ is an outer lipeomorphism, where $R=\left(\frac{1-r}{1+r}\right)^{\frac{1}{2}}$. Thus, $\widehat{\rho}\colon (\mathbb{S}^n,e_{n+1})\to (\R^n,0)$ is a germ of an outer lipeomorphism.

If $(X,d_{X,out})$ and $(Y,d_{X,out})$ are lipeomorphic at infinity.
By removing compact subsets of $X$ and $Y$, if necessary, we may assume that there is an outer lipeomorphism $\varphi \colon X\to Y$. By removing larger compact subsets, if necessary, we may assume that $X$ and $Y$ do not intersect $B_r^n(0)$ for some $r>0$. Therefore $\iota(X\setminus \{0\}) $ and $\iota(Y\setminus \{0\}) $ are subsets of $B_{\frac{1}{r}}^n(0)$. By the implication $(1)\Rightarrow (3)$, $\tilde \varphi= \iota\circ \varphi\circ \iota\colon \iota(X\setminus \{0\})\to \iota(Y\setminus \{0\}) $ is an outer lipeomorphism. Thus, $\widehat{\rho}^{-1}\circ\tilde \varphi \circ \widehat{\rho}=\rho^{-1}\circ\varphi \circ \rho\colon (\widehat{X}, e_{n+1})\to (\widehat{Y}, e_{n+1})$ is a germ of an outer lipeomorphism.

Reciprocally, if  $\psi\colon (\widehat{X}, e_{n+1})\to (\widehat{Y}, e_{n+1})$ is a germ of an outer lipeomorphism, $\tilde \varphi=\widehat{\rho}\circ \psi \circ \widehat{\rho}^{-1}\colon (\widehat{\rho}(\widehat{X}), 0)\to (\widehat{\rho}(\widehat{Y}),0)$ is also a germ of an outer lipeomorphism.  By the implication $(3)\Rightarrow (1)$, $ \iota\circ \tilde \varphi\circ \iota=\rho\circ \psi \circ \rho^{-1}\colon X\to Y $ is an outer lipeomorphism at infinity.
\end{proof}

Since the stereographic projection is an outer lipeomorphism far from $e_{n+1}$ and the inversion mapping is an outer lipeomorphism far from $e_{n+1}$ and infinity, as an easy consequence of Theorem \ref{thm:charac_outer_geo_lip}, we obtain the following:
\begin{corollary}\label{cor:charac_outer_geo_lip}
Let $X\subset\R^n$ and $Y\subset\R^m$ be unbounded sets.
Then, the following statements are equivalent:
\begin{enumerate}
 \item $X$ and $Y$ are outer lipeomorphic;
 \item The pointed sets $(\widehat{X},e_{n+1})$ and $(\widehat{Y},e_{m+1})$ are outer lipeomorphic;
 \item The sets $\iota(X\setminus \{0\})\cup \{0\}$ and $\iota(Y\setminus \{0\})\cup \{0\}$ are outer lipeomorphic. 
\end{enumerate}
\end{corollary}

\section{Outer and inner definable Lipschitz geometry: local vs. global}
\label{sec:out_inn_geo}
In this Section, we obtain some versions of Theorem \ref{thm:charac_outer_geo_lip} with stronger statements when the involved sets are definable.

\begin{corollary}
	\label{thm:charac_geo_lip}
Let $X\subset\R^n$ and $Y\subset\R^m$ be definable sets in $\mathcal{S}$ with connected links at infinity. Let $\sigma,\tilde\sigma\in \{inn,out\}$.
Then, the following statements are equivalent:
\begin{enumerate}
 \item There is a definable lipeomorphism at infinity $\varphi\colon (X,d_{X,\sigma})\to (Y,d_{X,\tilde\sigma})$ that preserves the outer distance to the origin;
 \item There is a germ of definable lipeomorphism $\psi\colon (\widehat{X},d_{\widehat{X},\sigma},e_{n+1})\to(\widehat{Y},d_{\widehat{Y},\tilde\sigma},e_{m+1})$ that preserves the last coordinate;
 \item There is a germ of lipeomorphism $\tilde\varphi\colon (\iota(X\setminus \{0\}), d_{\iota(X\setminus \{0\}),\sigma},0)\to (\iota(Y\setminus \{0\}), d_{\iota(Y\setminus \{0\}),\tilde\sigma},0)$ that preserves the outer distance to the origin. 
\end{enumerate}
\end{corollary}
\begin{proof}
By Propositions \ref{prop:arc_lip_criterion} and \ref{prop:local_arc_lip_criterion}, it is enough to consider the case where $\dim X=\dim Y=1$. But in that case, the distances $d_{X,out}$ and $d_{X,inn}$ are equivalent either at infinity or around $p\in X$.

Thus, we may assume that $\sigma=\tilde\sigma=out$ and that $X$ and $Y$ are definable curves with connected links at infinity and that are LNE. Now, the result follows from Theorem \ref{thm:charac_outer_geo_lip}.
\end{proof}

In fact, we can prove the following generalization of Corollary \ref{thm:charac_geo_lip}:

\begin{theorem}
\label{cor:compactification_eq_two}
Let $X\subset\R^n$ and $Y\subset\R^m$ be connected definable sets in $\mathcal{S}$ with connected links at infinity. Let $\sigma,\tilde\sigma\in \{inn,out\}$. 
Then, the following statements are equivalent:
\begin{enumerate}
 \item $(X,d_{X,\sigma})$ and $(Y,d_{X,\tilde\sigma})$ are lipeomorphic;
 \item The pointed sets $(\widehat{X},d_{\widehat{X},\sigma}, e_{n+1})$ and $(\widehat{Y},d_{\widehat{Y},\tilde\sigma}, e_{m+1})$ are lipeomorphic;
 \item The germs $(\iota(X\setminus \{0\}), d_{\iota(X\setminus \{0\}),\sigma},0)$ and $(\iota(Y\setminus \{0\}), d_{\iota(Y\setminus \{0\}),\tilde\sigma},0)$ are lipeomorphic. 
\end{enumerate}
\end{theorem}
\begin{proof}

Let $\varphi\colon (X,d_{X,\sigma})\to (Y,d_{X,\tilde\sigma})$ be a mapping. Let $\hat{\varphi}\colon (\widehat{X},d_{\widehat{X},\sigma}, e_{n+1})\to (\widehat{Y},e_{m+1})$ and $\tilde\varphi\colon(\iota(X\setminus \{0\}), d_{\iota(X\setminus \{0\}),\sigma},0)\to(\iota(Y\setminus \{0\}), d_{\iota(Y\setminus \{0\}),\tilde\sigma},0)$ be the mappings given by $\hat{\varphi}=\rho^{-1}\circ \varphi\circ \rho$ and $\tilde\varphi=\iota\circ \varphi\circ \iota$.

%
%
%

By Corollary \ref{cor:charac_LNE_infinity}, if $\{X_i\}_{i=1}^k$ (resp. $\{Y_i\}_{i=1}^d$) is a pancake decomposition of $X$ (resp. $Y$), then $\{\widehat{X}_i\}_{i=1}^k$ and $\{\iota(X_i\setminus\{0\})\}_{i=1}^k$ (resp. $\{\widehat{Y}_i\}_{i=1}^d$ and $\{\iota(Y_i\setminus\{0\})\}_{i=1}^d$) are respectively pancake decompositions of $\widehat{X}$ and $\iota(X\setminus\{0\})$ (resp. $\widehat{Y}$ and $\iota(Y\setminus\{0\})$).

If $\sigma=out$ (resp. $\tilde\sigma=out$), we set $\widetilde{X}=X$ (resp. $\widetilde{Y}=Y$) and $\mu_1\colon X\to \widetilde{X}$ (resp. $\mu_2\colon Y\to \widetilde{Y}$) is the identity mapping, and if $\sigma=inn$ (resp. $\tilde\sigma=inn$), then $\widetilde{X}$ (resp. $\widetilde{Y}$) is the Lipschitz normal embedding of $X$ (resp. $Y$) and $\mu_1\colon X\to \widetilde{X}$ (resp. $\mu_2\colon Y\to \widetilde{Y}$) is the definable inner lipeomorphism given by Theorem \ref{thm:Lip_normal_embeddins}.

Similarly, we have mappings $\hat{\mu}_1\colon \widehat{X}\to \widetilde{\widehat{X}}$, $\hat{\mu}_2\colon \widehat{Y}\to \widetilde{\widehat{Y}}$, $\tilde{\mu}_1\colon \iota(X\setminus\{0\})\to \widetilde{\iota(X\setminus\{0\})}$ and $\tilde{\mu}_2\colon \iota(Y\setminus\{0\})\to \widetilde{\iota(Y\setminus\{0\})}$.

Thus, we have the following:
\begin{enumerate}
 \item [i)]  $\varphi$ is a lipeomorphism if and only if $\psi\colon \widetilde{X} \to \widetilde{Y}$ given by $\psi=\mu_2^{-1}\circ \varphi \circ \mu_1$ is an outer lipeomorphism;
 \item [ii)]  $\tilde\varphi$ is a lipeomorphism if and only if $\tilde\psi\colon \widetilde{\iota(X\setminus\{0\})}\to \widetilde{\iota(Y\setminus\{0\})}$ given by $\tilde\psi=\tilde\mu_2^{-1}\circ \tilde\varphi \circ \tilde\mu_1$ is an outer lipeomorphism;
 \item [iii)]  $\hat\varphi$ is a lipeomorphism if and only if $\hat\psi\colon \widetilde{\widehat{X}} \to \widetilde{\widehat{Y}}$ given by $\hat\psi=\hat\mu_2^{-1}\circ \hat\varphi \circ \hat\mu_1$ is an outer lipeomorphism. 
\end{enumerate}

We assume that $\widetilde{X}\subset \R^N$ and $\widetilde{Y}\subset\R^M$.
By Theorem \ref{thm:charac_outer_geo_lip}, we have that the following statements are equivalent:
\begin{enumerate}
 \item $\widetilde{X}$ and $\widetilde{Y}$ are outer lipeomorphic at infinity;
 \item The germs of the stereographic modifications $(\widehat{\widetilde{X}},e_{N+1})$ and $(\widehat{\widetilde{Y}},e_{M+1})$ are outer lipeomorphic;
 \item The germs of the inversions $(\iota(\widetilde{X}\setminus \{0\}),0)$ and $(\iota(\widetilde{X}\setminus \{0\}),0)$ are outer lipeomorphic. 
\end{enumerate}

It follows from Rademacher's theorem that a homeomorphism $\mu\colon A_1\to A_2$, between two definable sets $A_1\subset \R^{n_1}$ and $A_2\subset \R^{n_2}$, is an inner lipeomorphism if and only if there exists a constant $K\geq 1$ such that the derivatives $D\mu_x$ and $D\mu^{-1}_y$ exist, respectively, almost everywhere on $A_1$ and $A_2$ and, moreover, $\|D\mu_x\|\leq K$ and $\|D\mu^{-1}_y\|\leq K$. 
Thus, like it was done in the proof of Proposition \ref{prop:arc_lip_criterion} for the outer distance and by the proof of Theorem \ref{thm:charac_outer_geo_lip}, we obtain that $\mu\colon A_1\to A_2$ is an inner lipeomorphism at infinity if $\iota \circ \mu \circ \iota\colon (\iota(A_1\setminus\{0\}),0)\to (\iota(A_2\setminus\{0\}),0)$ is the germ of an inner lipeomorphism if and only if $\rho^{-1} \circ \mu\circ \rho\colon (\widehat{A}_1,e_{n_1+1})\to (\widehat{A}_2,e_{n_2+1})$ is the germ of an inner lipeomorphism.

Therefore, the proof follows from the above equivalences.

\end{proof}

Since the stereographic projection is an outer lipeomorphism far from $e_{n+1}$ and the inversion mapping is an outer lipeomorphism far from $e_{n+1}$ and infinity, as an easy consequence of Theorem \ref{cor:compactification_eq_two}, we obtain the following:
\begin{corollary}
\label{cor:global_compactification_eq_two}
Let $X\subset\R^n$ and $Y\subset\R^m$ be connected definable sets in $\mathcal{S}$ with connected links at $0$. Let $\sigma,\tilde\sigma\in \{inn,out\}$. 
Then, the following statements are equivalent:
\begin{enumerate}
 \item $(X,d_{X,\sigma})$ and $(Y,d_{X,\tilde\sigma})$ are lipeomorphic;
 \item The pointed sets $(\widehat{X},d_{\widehat{X},\sigma}, e_{n+1})$ and $(\widehat{Y},d_{\widehat{Y},\tilde\sigma},e_{m+1})$ are lipeomorphic;
 \item $(\iota(X\setminus \{0\}), d_{\iota(X\setminus \{0\}),\sigma})$ and $(\iota(Y\setminus \{0\}), d_{\iota(Y\setminus \{0\}),\tilde\sigma})$ are lipeomorphic. 
\end{enumerate}
\end{corollary}

\bigskip

\noindent{\bf Acknowledgements}. The author would like to thank Alexandre Fernandes and Euripedes da C. da Silva for their interest and incentive in this research.

\end{document}